\documentclass[makeidx]{amsart}

\usepackage{a4wide}
\usepackage{amssymb,amscd}
\usepackage[utf8]{inputenc}
\usepackage{pb-diagram}
\usepackage{graphicx}
\usepackage{color}
\usepackage[hyperref]{hyperref}

\renewcommand{\phi}{\varphi}
\newcommand{\C}{{\mathbb{C}}}
\newcommand{\N}{{\mathbb{N}}}
\newcommand{\R}{{\mathbb{R}}}

\newcommand{\Z}{{\mathbb{Z}}}
\newcommand{\1}{{\mathbf{1}}}

\renewcommand{\epsilon}{\varepsilon}
\renewcommand{\theta}{\vartheta}
\renewcommand{\S}{{\mathbb{S}}}
\newcommand{\T}{{\mathbb{T}}}

\newcommand{\lie}[1]{{\mathcal{L}_{#1}}}
\newcommand{\pairing}[2]{{\langle{#1}|{#2}\rangle}}

\newcommand{\norm}[1]{{\lVert #1\rVert}}
\newcommand{\abs}[1]{{\left\lvert #1\right\rvert}}

\newcommand{\z}{{\mathbf{z}}}

\DeclareMathOperator{\Ad}{Ad}

\DeclareMathOperator{\Crit}{Crit}

\newcommand{\g}{{\mathfrak{g}}}

\DeclareMathOperator{\Hessian}{Hess}

\DeclareMathOperator{\rank}{rank}

\DeclareMathOperator{\stab}{Stab}

\theoremstyle{plain}
\newtheorem{theorem}{Theorem}
\newtheorem{proposition}[theorem]{Proposition}
\newtheorem{corollary}[theorem]{Corollary}

\newtheorem{lemma}{Lemma}

\theoremstyle{remark}

\theoremstyle{definition}
\newtheorem*{defi}{Definition}

\numberwithin{equation}{section}

\begin{document}

\title[Desingularisation of symplectic orbifolds]{Desingularisation of
  orbifolds obtained from symplectic reduction at generic coadjoint
  orbits}

\author[K.~Niederkrüger]{Klaus Niederkrüger}
\author[F.~Pasquotto]{Federica Pasquotto}

\email[K.\ Niederkrüger]{kniederk@umpa.ens-lyon.fr}
\email[F.~Pasquotto]{pasquott@few.vu.nl}

\address[K.\ Niederkrüger]{École Normale Supérieure de Lyon\\
  Unité de Mathématiques Pures et Appliquées\\
  UMR CNRS 5669\\
  France}

\address[F.~Pasquotto]{Department of Mathematics, Faculty of Sciences\\
  Vrije Universiteit\\De Boelelaan 1081a\\1081 HV Amsterdam\\
  The Netherlands}

\begin{abstract}
  We show how to construct a resolution of symplectic orbifolds
  obtained as quotients of presymplectic manifolds with a torus
  action.  As a corollary, this allows us to desingularise generic 
  symplectic
  quotients. Given a manifold with a Hamiltonian action of a compact
  Lie group, symplectic reduction at a coadjoint orbit which is
  transverse to the moment map produces a symplectic orbifold. If
  moreover the points of this coadjoint orbit are regular elements of
  the Lie coalgebra, that is, their stabiliser is a maximal torus, the
  result for torus quotients may be applied to obtain a
  desingularisation of these symplectic orbifolds.  Regular elements
  of the Lie coalgebra are generic in the sense that the singular
  strata have codimension at least three.
  
  Additionally, we show that even though the result of a symplectic cut
  is an orbifold, it can be modified in an arbitrarily small neighbourhood 
  of the cut hypersurface to obtain a smooth symplectic manifold.
 
\end{abstract}

\maketitle

\section{Introduction}

Let $G$ be a compact Lie group, and let $(W,\omega)$ be a Hamiltonian
$G$--manifold with moment map $\mu:\,W\to \g^*$. If $\mu$ is
transverse to the coadjoint orbit $\Ad(G)^*\nu$, then the preimage
$\mathcal{O}(\Ad(G)^*\nu) := \mu^{-1}(\Ad(G)^*\nu)$ of the coadjoint
orbit is a smooth $G$--invariant submanifold.  It also follows that
every element on the coadjoint orbit $\Ad(G)^*\nu$ is a regular value
of the moment map $\mu$, and so the $G$--action is locally free,
because for $X\in\g$ the infinitesimal generator $X_W$ satisfies the
equation $i_{X_W}\omega = \pairing{d\mu}{X}$, so that $X_W$ cannot
vanish at regular points of the moment map $\mu$.  The quotient
\begin{equation*}
  \mathcal{O}(\Ad(G)^* \nu)/\!/G :=  \mu^{-1}(\Ad(G)^*\nu)/G
\end{equation*}
is an orbifold.  The easiest way to define the canonical symplectic
structure on $\mathcal{O}(\Ad(G)^* \nu)/\!/G$ is by using the
isomorphism
\begin{equation}\label{eqn: quotientiso}
  \mathcal{O}(\Ad(G)^* \nu)/\!/G \cong \mu^{-1}(\nu)/\stab(\nu) \;.
\end{equation}
The $2$--form $\omega$ descends to a non-degenerate form on the
quotient on the right hand side, because the tangent space to the 
$\stab(\nu)$--orbits spans
the kernel of $\omega|_{T\mu^{-1}(\nu)}$.

In this paper we prove the following result, which is a generalisation
of the result in \cite{NiederkruegerPasquottoCyclic}.

\begin{theorem}\label{main}
  Let $(M,\omega)$ be a symplectic orbifold arising as the quotient of
  a presymplectic $\T^k$--manifold by the given torus action and let
  $X$ be the set of orbifold singularities of $M$.  Assume that $X$ is
  compact: then for any neighbourhood $U$ of $X$ there exists a
  symplectic resolution of $M$ supported on~$U$.
\end{theorem}

\begin{defi}
  An element $\nu$ in the Lie coalgebra $\g^*$ is called
  \textbf{regular} if its stabiliser with respect to the coadjoint
  action is a maximal torus.
\end{defi}

In fact most elements of the Lie coalgebra are regular: their
complement consists of the union of finitely many subspaces of $\g^*$,
the singular strata, which have codimension at least $3$.  Using the
isomorphism \eqref{eqn: quotientiso} we can apply the theorem to a
symplectic reduction at a generic element $\nu\in\g^*$ to obtain the
following

\begin{corollary}
  Let $G$ be a connected compact Lie group, and let $(W,\omega)$ be a
  Hamiltonian $G$--manifold with moment map $\mu$.  Choose an element
  $\nu\in\g^*$ that is a regular value of $\mu$ and a regular element
  of the Lie coalgebra.  The reduced space $M := \mathcal{O}(\Ad(G)^*
  \nu)/\!/G$ is a symplectic orbifold and, provided that the set $X$
  of orbifold singularities is compact, it admits a symplectic
  resolution supported on an arbitrarily small neighbourhood of $X$.
\end{corollary}

\subsubsection*{Acknowledgements}
K.\ Niederkrüger is working at the \emph{ENS de Lyon}, where he is
being funded by the Agence Nationale de la Recherche (ANR) project
\emph{Symplexe}.  F.\ Pasquotto is working at the \emph{Vrije
  Universiteit Amsterdam} and is supported by VENI grant 639.031.620
of the Nederlandse Organisatie voor Wetenschappelijk Onderzoek (NWO).

We would like to thank the \textit{Centre International de Rencontres
  Mathématiques} at Luminy for their hospitality in the summer of
2008, while we were working on the main ideas for this paper.  We are
also grateful to Eugene Lerman for his patience in answering our
questions and for asking in turn more questions that stimulated us to
think further.

\section{Presymplectic manifolds and symplectic orbifolds}

In this section, $G$ will denote a compact connected Lie group and
$\g$ its Lie algebra.  If $G$ acts on a manifold $P$, then there
exists a homomorphism
\begin{equation*}
  \g\to \mathfrak{X}(P)\;,\quad X\mapsto X_P(p)= 
  \left.\frac{d}{dt}\right|_{t=0}\exp(tX)*p\;.
\end{equation*}
The vector $X_P$ is called the \textbf{infinitesimal action} of
$X\in\g$.

\begin{defi}
  Let $P$ be a $(2n+k)$--dimensional smooth manifold and assume it
  admits a locally free $G$--action and a closed $2$--form $\omega_P$
  such that:
  \begin{itemize}
  \item[(i)] $\dim G = k$;
  \item[(ii)] $\omega_P^n\neq 0$;
  \item[(iii)] the infinitesimal action of $\g$ spans the kernel of
    $\omega_P$.
  \end{itemize}
  We call $P$ a \textbf{presymplectic $G$--manifold}.
\end{defi}

\begin{defi}
  A \textbf{symplectic orbifold} $M$ is a Hausdorff, second countable
  topological space, equipped with an atlas of uniformizing charts
  $(\widetilde{U}_i,\Gamma_i, \phi_i, \omega_i)$, where
  $\widetilde{U}_i$ is an open connected subset of $\R^{2n}$,
  $\Gamma_i$ is a finite group of symplectomorphisms of
  $(\widetilde{U}_i, \omega_i)$ and $\phi_i:\, \widetilde{U}_i \to M$
  induces a homeomorphism from $\widetilde{U}_i/\Gamma_i$ to $U_i
  \subset M$.  These charts are required to cover $M$ and to satisfy
  the following compatibility condition: if $x\in \widetilde{U}_i$ and
  $y\in \widetilde{U}_j$ are such that $\phi_i(x)=\phi_j(y)$ then
  there exists a symplectomorphism from a neighbourhood of $x$ onto a
  neighbourhood of $y$ whose composition with $\phi_j$ is $\phi_i$.
\end{defi}

\begin{defi}
  Let $M_1$ and $M_2$ be orbifolds with uniformizing charts
  $\bigl\{(\widetilde{U}_i,\Gamma_i, \phi_i)\bigr\}$ and
  $\bigl\{(\widetilde{V}_j,\Lambda_j, \psi_j)\bigr\}$, respectively.
  A \textbf{smooth orbifold map} $f:\, M_1 \to M_2$ is a continuous
  map such that if $x \in\widetilde{U}_i$ and $y\in\widetilde{V}_j$
  are such that $f\bigl(\phi_i(x)\bigr) = \psi_j(y)$, then there
  exists a smooth map $\widetilde f$ from a neighbourhood of $x$ to a
  neighbourhood of $y$ such that $\psi_j\circ \widetilde f = f\circ
  \phi_i$.
\end{defi}

A $G$--orbifold $M$ is a smooth orbifold map $G\times M \to M$ with
the usual properties of an action.  The definition of presymplectic
$G$--orbifold is also analogous to the corresponding one for
manifolds.

\begin{lemma}\label{quotient_by_a_product}
  Let $P$ be a presymplectic $G_1\times G_2$--manifold. Then $P/G_1$
  is a presymplectic $G_2$--orbifold.  Moreover, $P/(G_1\times G_2)$
  and $(P/G_1)/G_2$ are isomorphic as symplectic orbifolds.
\end{lemma}

The proof of the lemma above is analogous to the case of smooth manifolds
with free group actions. In our situation, though, one has to use the 
symplectic orbifold slice theorem, as stated in \cite{LermanTolman}.

\subsection{Symplectic resolutions}

\begin{defi}
  Let $M$ be an orbifold with singular set $X$.  A \textbf{resolution
    of $M$} consists of a smooth manifold $\widetilde{M}$ and a
  continuous surjective map $p:\, \widetilde{M}\to M$ which is a
  diffeomorphism on the complement of the singular set.

  If $(M, \omega)$ is a symplectic orbifold (presymplectic
  $G$--orbifold) and $U$ is a neighbourhood of the singular set $X$,
  then a \textbf{symplectic (presymplectic) resolution of $M$
    supported on $U$} consists of a smooth symplectic (presymplectic
  $G$--) manifold $(\widetilde{M}, \widetilde{\omega})$ and a
  ($G$--equivariant) resolution $p:\, \widetilde{M}\to M$ such that
  $p^*\omega=\widetilde{\omega}$ on the complement of $U$.
\end{defi}

\section{Construction of the resolution}\label{Sec:Construction}

\subsection{Strategy}

Let $M$ be a symplectic orbifold arising as the quotient of a
presymplectic $\T^k$--manifold $P$ by the given group action, and let
$U$ be a neighbourhood of the set of orbifold singularities of $M$. If
we split off an $\S^1$--factor from the $k$--dimensional torus and
view $M$ as the quotient of a presymplectic
$\S^1\times\T^{k-1}$--manifold, by Lemma~\ref{quotient_by_a_product}
we have that $P/\S^1$ is a presymplectic $\T^{k-1}$--orbifold and
$(P/\S^1)/\T^{k-1}\cong M$.  If the $\S^1$--action on $P$ happens to
be free, then $P/\S^1$ is a smooth presymplectic $\T^{k-1}$--manifold
and we can split off another $\S^1$--factor. If the action is only
locally free, in Lemma~\ref{equivariant_circle_resolution} below we
show that one can construct a presymplectic resolution $P_1$ of
$P/\S^1$ (with $\T^{k-1}$--action) supported on an arbitrarily small
neighbourhood $U_1$ of the set of singular points of $P/\S^1$. In
particular, we can choose $U_1$ in such a way that its image under the
$\T^{k-1}$--action is contained in $U$.  From this it follows that the
map $p_1:\, P_1\to P/\S^1$ induces a surjective map $[p_1]:\,
P_1/\T^{k-1}\to (P/\S^1)/\T^{k-1} \cong M$ which is a diffeomorphism
outside the set of singular points of $M$ and a symplectomorphism
outside~$U$.

If we iterate this step (see Fig.~\ref{fig: resolution_scheme}), we
get a sequence of manifolds $P_1,\dotsc,P_k$ such that:
\begin{itemize}
\item[(i)] $P_i$ is a presymplectic $\T^{k-i}$--manifold;
\item[(ii)] $p_i:\, P_i\to P_{i-1}/\S^1$ is a presymplectic resolution
  supported on $U_i$ with an action of $\T^{k-i}$ such that
  $U_i/\T^{k-i} \subset ([p_1]\circ \dotsm \circ [p_{i-1}])^{-1} (U)$;
\item[(iii)] the induced maps $[p_i]:\, P_i/\T^{k-i}\to
  P_{i-1}/\T^{k-i+1}$ can be composed to obtain a symplectic
  resolution $[p] = [p_1]\circ[p_2]\circ\dotsm\circ [p_k]:\,
  \widetilde{M}=P_k\to M=P/\T^k$ of $M$ supported on $U$.
\end{itemize} 

\begin{figure}
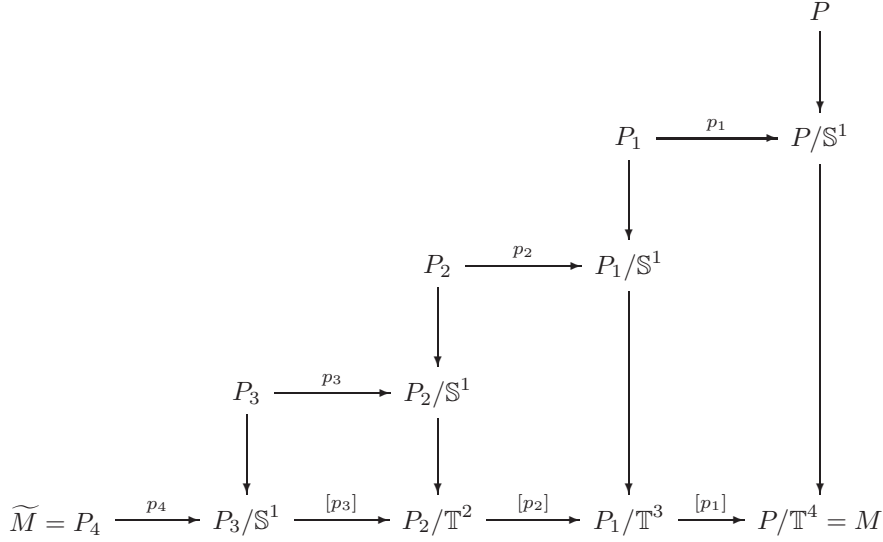

  \centering
  \begin{equation*}
    \begin{diagram} 
      \node[5]{P} \arrow{s} \\
      \node[4]{P_1} \arrow{e,t}{p_1} \arrow{s} \node{P/\S^1} \arrow[3]{s} \\
      \node[3]{P_2} \arrow{e,t}{p_2} \arrow{s} \node{P_1/\S^1} \arrow[2]{s} \\
      \node[2]{P_3} \arrow{e,t}{p_3} \arrow{s} \node{P_2/\S^1} \arrow{s} \\
      \node{\widetilde{M} = P_4} \arrow{e,t}{p_4} \node{P_3/\S^1}
      \arrow{e,t}{[p_3]} \node{P_2/\T^{2}} \arrow{e,t}{[p_2]}
      \node{P_1/\T^{3}} \arrow{e,t}{[p_1]} \node{P/\T^4 = M}
    \end{diagram}
  \end{equation*}
  
  \caption{Resolution scheme applied to a presymplectic
    $\T^4$--manifold $(P,\omega)$.}
  \label{fig: resolution_scheme}
\end{figure}
\begin{lemma}\label{equivariant_circle_resolution}
  Let $(P,\omega)$ be a presymplectic $\S^1\times G$--manifold.  Let
  $U$ be a neighbourhood of the set of singular points for the
  $\S^1$--action.  Then we find a presymplectic $\S^1\times
  G$--manifold $(\widetilde P,\widetilde\omega)$ such that the
  $\S^1$--action is free, and such that there is a map $p:\,\widetilde
  P/\S^1 \to P/\S^1$ that is a $G$--equivariant presymplectic
  resolution of $P/\S^1$ supported on $U$.
\end{lemma}

The rest of Section~\ref{Sec:Construction} will be devoted to the
proof of this lemma.  For this we have to adapt the symplectic cut
\cite{Lerman_SymplecticCuts} to our presymplectic setting.  Roughly
speaking we construct an auxiliary circle action in a tubular
neighbourhood of the stratum of singularities with highest isotropy
group.  Then we cut out a smaller neighbourhood, and using the
auxiliary action we collapse the boundary of the cavity to reobtain a
closed manifold.  With this construction, we have removed the points
with highest isotropy groups, and then we can reapply the same steps
to reduce successively the order of the worst singular points, until
none are left.

\subsection{Reducing the order of the singularities of the
  $\S^1$--action}

Choose an $\S^1\times G$--invariant splitting of the tangent bundle,
\begin{equation*}
  TP = \ker\omega_P \oplus \Omega^P\;.
\end{equation*}
Since $\Omega^P$ is a symplectic subbundle, it admits an $\S^1\times
G$--invariant almost complex structure $J$ that is compatible with
$\left.\omega_P\right|_{\Omega^P}$ (see for example
\cite[Section~5.5]{McDuffSalamonIntro}).  Consider the metric on
$\Omega^P$ given by $g=\omega_P(J-,-)$, and extend it to an invariant
metric on $TP$ such that $\ker\omega_P \perp \Omega^P$.

Consider the stratification $\{P_k\}$ of the singular set of the
$\S^1$--action on $P$, with $P_k=\bigl\{p\in P \bigm|\,
\stab_{\S^1}(p)\cong \Z_k\bigr\}$.  If $k$ is maximal, that is,
$P_{k'}=\emptyset$ for $k'>k$, then $P_k$ is a closed, $\S^1\times
G$--invariant submanifold of $P$, of codimension at least~$2$.  By
restricting to one component, we may further assume $P_k$ to be
connected. A model of a neighbourhood of $P_k$ is given by a
neighbourhood of the zero section in the total space of its normal
bundle $\nu_k$.  The linearisation of the $\S^1\times G$--action on
$P$ defines an action on $\nu_k$ that is equivalent to the given one
on $P$ (in the sense that the exponential map defines an equivariant
diffeomorphism from a neighbourhood of the zero section in $\nu_k$ to
a neighbourhood of $P_k$ in~$P$).  Later it will be necessary to
introduce an auxiliary circle action.  To avoid confusions, from now
on we will call the given $\S^1$--action the $\beta$--action and we
will write it as $\lambda *_\beta v$ for $\lambda\in\S^1$
and~$v\in\nu_k$, whereas we will write $g*_G v$ for the $G$--action.

Let $x\in P_k$ be a point in the minimal stratum.  The
$\S^1$--stabiliser $\stab_{\S^1}(x)\cong\Z_k$ acts by isometric
$J$--linear transformations on $\Omega^P_x$, hence there exists an
isomorphism between the Hermitian vector spaces $(\Omega^P_x, J, g)$
and the standard Hermitian space $\C^n$ such that the linearised
$\Z_k$--action takes the form
\begin{equation}\label{normal_form}
  \lambda*_\beta \z = (\lambda^{\widetilde{a}_1}z_1,\dotsc,
  \lambda^{\widetilde{a}_n}z_n),
  \quad \lambda\in\Z_k, \; \z\in\C^n,
  \;\widetilde{a}_1,\dotsc,\widetilde{a}_n\in\Z\;.
\end{equation}
Without loss of generality we may assume that $0 = \widetilde{a}_1 =
\dotsb = \widetilde{a}_m < \widetilde{a}_{m+1}\leq \dotsb \le
\widetilde{a}_n < k$ for some $m$.  The first $m$ directions span the
space $\Omega^P_x\cap T_xP_k$, and since the others are orthogonal,
they coincide with the fibre $\nu_k(x)$.  In particular, it follows
that $\nu_k$ is a $J$--complex bundle with a fibrewise $J$--linear
$\Z_k$--action that commutes with the $\S^1\times G$--action.

Denote by $a_1<\dotsb <a_l$ the \emph{distinct} exponents occurring in
the normal form \eqref{normal_form} for the action: $\nu_k$ splits
thus into a direct sum of $\S^1\times G$--invariant subbundles
\begin{equation*}
  \nu_k = E_1\oplus\dotsb \oplus E_l \;,
\end{equation*}
where $E_i(x)$ denotes the eigenspace corresponding to the eigenvalue
$\lambda^{a_i}$ in the fibre at the point $x$.  The splitting is well
defined for each component of $P_k$.  This allows us to extend the
$\Z_k$--action to an auxiliary circle action $\phi$ by setting for any
$\lambda \in\S^1$
\begin{equation*}
  \lambda *_\phi v  := \lambda^{a_1}v_1+ \dotsb +\lambda^{a_l}v_l \;,
\end{equation*}
where $v=v_1+\dotsb +v_l$ is a splitting with respect to the
eigenspaces defined above.  This $\S^1$--action is fibrewise and
$J$--linear, and commutes with the original $\S^1\times G$--action.
The presymplectic form $\omega_P$ is $\phi$--invariant at points of
$P_k$, but unfortunately it does not need to be invariant at other
points. By averaging $\omega_P$ over the $\phi$--action, we obtain a
closed $2$--form $\omega$ which is invariant with respect to both the
$\S^1\times G$-- and the $\phi$--action, and such that the $\S^1\times
G$--orbits still lie in the kernel of $\omega$.  At points of $P_k$,
where $\omega_P$ was already $\phi$--invariant, we did not change it
by averaging, and so we have $\omega^n\neq0$.  It follows that there
is a small neighbourhood of the zero section, where $\omega$ will be
$\S^1\times G$--presymplectic.

\begin{proposition}\label{propo: presymplectic neighbourhood theorem}
  There exist neighbourhoods $U_1,U_2$ of $P_k$ in $\nu_k$ and an
  $\S^1\times G$--equivariant diffeomorphism $\Psi:\, U_1\to U_2$ such
  that
  \begin{equation*}
    \Psi^*\omega = \omega_P\;.
  \end{equation*}
\end{proposition}

The proof of this statement is a variation on the proof of the
analogous statement in the symplectic case, but for completeness we
have included it in the appendix.  This proposition shows that we can
pull-back the $\phi$--action to a neighbourhood of $P_k$ where it
gives us an auxiliary action for which the original presymplectic form
$\omega_P$ is invariant.  Equivalently we can work with the action
$\phi$ we have defined above and the averaged symplectic form
$\omega$. In this paper we choose to do the latter.

\begin{proposition}\label{propo: Morse Bott}
  There exists a neighbourhood $U$ of $P_k$ in $\nu_k$ and a non
  negative $\S^1\times G$--invariant Morse-Bott function $\mu_\phi:\,
  U \to \R$ such that
  \begin{itemize}
  \item $i_{X_\phi}\omega = d\mu_\phi$
  \item $\mu_\phi$ vanishes only on the zero section of $\nu_k$, and
    it is strictly increasing in radial fibre direction.
  \end{itemize}
\end{proposition}
\begin{proof}
  Since $\omega$ is $\phi$--invariant, one has that $di_{X_\phi}\omega
  = \lie{X_\phi}\omega = 0$.  For the time being, let $U$ be any
  tubular neighbourhood of $P_k$, where $\omega$ is defined.  The
  closed $1$--form $i_{X_\phi}\omega$ represents a class in $H^1(U)$
  which vanishes if we pull it back to the zero section $P_k$: Given
  that $H^1(U) \cong H^1(P_k)$, it follows that $i_{X_\phi}\omega$ is
  exact on $U$, i.e., there exists a function $\mu_\phi$ such that
  $i_{X_\phi}\omega = d\mu_\phi$.  The function $\mu_\phi$ is uniquely
  defined up to an additive constant (which we may choose such that
  $\mu_\phi \equiv 0$ on $P_k$) and is $\S^1\times G$--invariant.
  
  Recall that a function $f:\, M\to \R$ is called Morse-Bott if
  $\Crit(f)$ is a submanifold of $M$ and $T_x\Crit(f) = \ker
  \Hessian_x (f)$ for all $x\in\Crit(f)$, where $\Hessian_x(f): \,
  T_xM \to T_xM$ denotes the Hessian of $f$ at the point $x$.

  In the case we are considering $\Crit(\mu_\phi) = P_k$: in fact, if
  $v\in P_k$, $X_\phi(v) = 0$, and hence $d\mu_\phi = i_{X_\phi}\omega
  = 0$.  Conversely, since $X_{\phi}$ always lies in the fibre of
  $\nu_k$ and $\omega$ restricts to a symplectic form there,
  $d\mu_\phi = i_{X_\phi}\omega$ can only vanish if $X_\phi=0$.  It is
  easy to show that if $v\in P_k$, the inclusion $T_v \Crit (\mu_\phi)
  \leq \ker \Hessian_v(\mu_\phi)$ holds.

  To see that equality holds one needs to show that $\dim\ker
  \Hessian_v (\mu_\phi) \leq \dim P_k$ or, equivalently, that
  $\rank\Hessian_v \mu_\phi$ is at least equal to the rank of $\nu_k$.
  Restricting $\omega$, $\phi$ and $\mu_\phi$ to one fibre of $\nu_k$
  we are in a proper symplectic situation and we may conclude that
  $\left.\mu_\phi\right|_{\nu_k(x)}$ is Morse (see
  \cite[Section~5.5]{McDuffSalamonIntro}), hence in particular it has
  full rank.  Introducing bundle coordinates on $U$ and computing the
  matrix of second derivatives of $\mu_\phi$, which represent the
  Hessian in these coordinates, we see that it always contains a non
  singular block, corresponding to the above restriction of $\mu_\phi$
  to one fibre, having rank equal to the rank of $\nu_k$.  This proves
  that $\mu_\phi$ is Morse-Bott, and it only remains to show that it
  is positive outside $P_k$.

  Let $\partial_r$ be the radial vector field on $\nu_k$ given by
  \begin{equation*}
    \partial_r(v) = \left.\frac{d}{dt}\right|_{t=1} t\cdot v
  \end{equation*}
  for $v\in \nu_k$.  We will show that $\mu_\phi$ strictly increases
  in radial direction, more precisely that $\lie{\partial_r} \mu_\phi
  \ge 0$ in $U$ (possibly after shrinking $U$) with equality only at
  the zero section.  By definition of $\mu_\phi$, one has
  $i_{\partial_r} d\mu_\phi = \omega (X_\phi, \partial_r)$, so it will
  suffice to show that there exists a neighbourhood of $P_k$ where
  $\omega(X_\phi, \partial_r) \ge 0$.  With $\pi$ denoting the bundle
  projection $\nu_k\to P_k$, the vertical bundle $V(\nu_k)$ of $\nu_k$
  can be identified with the pull-back
  \begin{equation*}
    \pi^* \nu_k = \bigl\{ (v,w) \in \nu_k \times \nu_k \bigm| \,
    \pi(v) = \pi(w)\bigr\} \;.
  \end{equation*}
  The identification of $\pi^*(\nu_k)$ and $V(\nu_k)$ goes as follows
  \begin{equation*}
    \pi^*(\nu_k) \to V(\nu_k), \quad (v,w) \mapsto
    \left.\frac{d}{dt}\right|_{t=0} (v+ tw) \;.
  \end{equation*}
  Let $v\in \nu_k$, and write it as $v = v_1 + \dotsm + v_l$ with
  respect to the splitting $\nu_k = E_1 \oplus \dotsm \oplus E_l$.
  The vectors $X_\phi$ and $\partial_r$ lie in $V(\nu_k) \cong
  \pi^*(\nu_k)$ and they are given by
  \begin{equation*}
    X_\phi(v) = (v, a_1 J v_1 + \dotsm + a_lJv_l) \quad \text{ and }\quad
    \partial_r(v) = (v,v) \;.
  \end{equation*}
  Let $w$ be another vector in $\nu_k$ with $\pi(v)=\pi(w)$,
  $\norm{w}=1$, and $w=w_1 + \dotsm + w_l$ with respect to the
  splitting $\nu_k = E_1 \oplus \dotsm \oplus E_l$.  We will show that
  if $v$ lies in the zero section of $\nu_k$, then one has
  \begin{equation}\label{eqn: positivity}
    \omega \Bigl((v,\sum_{j=1}^l a_j Jw_j), (v,w)\Bigr)>0 \;.
  \end{equation}
  Then by continuity there exists a neighbourhood of the zero section
  of $\nu_k$ where this holds for all $\norm{w}=1$ and therefore, by
  scaling, for all $w\neq 0$.  Hence in particular $\omega(X_\phi,
  \partial_r) \geq 0$ on this neighbourhood, with equality only at the
  zero section.  In order to prove \eqref{eqn: positivity}, recall that
  at the zero section of $\nu_k$ the differential of $\exp$ is the
  identity and the presymplectic form $\omega_P$ is left unchanged by
  averaging.  So it follows that
  \begin{equation*}
    \omega \Bigl((v,\sum_{j=1}^l a_j Jw_j), (v,w)\Bigr)
    =\omega\Bigl(\sum_{j=1}^l a_j Jw_j, w\Bigr)=
    \sum_{j=1}^l a_j\, \omega\bigl(Jw_j, w_j\bigr) > 0 \;,
  \end{equation*} 
  since the eigenspaces $E_j$'s are $\omega$--orthogonal.
\end{proof}

\subsection{Surgery along the minimal stratum}

As in the previous section, consider the minimal singular stratum
$P_k$, denote by $\nu_k$ its normal bundle in $P$ and take now the
product $\nu_k\times \C$.  We can extend the original $\beta$-- and
$G$--action to this product by letting $\S^1$ and $G$ act trivially on
the $\C$--factor, namely, for $v\in \nu_k(x)$, $w\in\C$
\begin{equation*}
  \lambda *_\beta (v,w) :=(\lambda *_\beta v, w) \qquad
  \textrm{and} \qquad g *_G (v,w) :=(g *_G v, w)\;.
\end{equation*}
We can define a second circle action on $\nu_k\times \C$ by setting
\begin{equation*}
  \lambda *_\phi (v,w) = (\lambda *_\phi v, \lambda^{-k}w)  =
  \bigl((\lambda^{a_1}v_1+ \dotsb +\lambda^{a_l}v_l),
  \lambda^{-k}w\bigr) \;,
\end{equation*}
where $v=v_1+\dotsb +v_l$ is the splitting with respect to the
eigenspaces defined above.  The $\beta$-- and $\phi$--actions commute
and therefore we can combine them and define a new $\S^1$--action
\begin{equation*}
  \lambda *_\tau (v,w) :=\lambda *_\phi \left(\lambda^{-1} *_\beta
    (v,w)\right) \;.
\end{equation*}
This $\tau$--action is not effective, because the $\phi$-- and the
$\beta$--action coincide for elements in $\Z_k$.  But consider the
short exact sequence
\begin{equation*}
  0\to \Z_k\to\S^1\to \hat\S^1 \to 0 \;,
\end{equation*}
with the homomorphism of the circle given by $\lambda\mapsto
\lambda^k$, and let $\hat{\S}^1$ act on $\nu_k\times \C$ by
$\sigma*_{\hat\tau} (v,w)=\lambda*_\tau (v,w)$ for some $\lambda\in
\S^1$ such that $\lambda^k=\sigma$.  This new action, which we denote
by $\hat{\tau}$, is not only effective but even free and the quotient
$(\nu_k\times \C)/\hat{\tau}$ is a smooth manifold.  Since $\phi$ and
$G$ commute with $\hat{\tau}$ and with each other, they descend to
this quotient, inducing an $\S^1\times G$--action, that is again
locally free outside the set $\bigl\{[v,0] \bigm|\, v\in P_k\bigr\}$.
To see this, assume there is $[v,w]\in (\nu_k\times\C)/\hat{\tau}$ and
a sequence $\bigl\{(\lambda_n,g_n)\bigr\}$ in $\S^1\times G$,
converging to the identity element $(1,e)$, such that
$(\lambda_n,g_n)*_{\phi\times G} [v,w]=[v,w]$ for all $n\in \N$.  This
means that for each $n$ there exists a unique element $\sigma_n\in
\S^1$ such that $\sigma_n\to 1$ and
\begin{equation}\label{equation:locally_free}
  (\lambda_n,g_n)*_{\phi\times G}(v,w)=\sigma_n*_{\hat{\tau}}(v,w) \;.
\end{equation}
By definition of the $\hat{\tau}$--action, $\sigma_n*_{\hat{\tau}}
(v,w) = \mu_n*_\tau(v,w)$ for some sequence $\{\mu_n\in\S^1\}$ such
that $\mu_n^k=\sigma_n$ and $\mu_n\to 1$. Dropping the $w$--coordinate
and projecting onto the zero section of $\nu_k$ we obtain
$g_n*_G\pi(v)=\mu_n*_\beta\pi(v)$, that is, $(\mu_n^{-1},g_n) \in
\stab_{\beta\times G}(\pi(v))$. Since the $\beta\times G$--action is
locally free, it follows that $\mu_n=1$ and $g_n=e$ for all $n$
sufficiently large.  We can now rewrite
Equation~\eqref{equation:locally_free} as $\lambda_n*_\phi(v,w)=
(v,w)$ and from this conclude that either $\lambda=1$ or $(v,w)$ has
to lie in $P_k\times\{0\}$.

We define a $2$--form $\Omega = (\omega, -i\,dw\wedge d\bar{w})$ on
$\nu_k\times \C$, which is invariant with respect to the
$\hat\tau$--action.  By construction, the infinitesimal generator of
this action can be written as $X_{\hat\tau} =
\frac{1}{k}(-X_{\beta}+X_{\phi})$.  The ``Hamiltonian'' for the
$\hat{\tau}$--action, given by
\begin{equation*}
  H_{\hat{\tau}} (v,w) = \frac{1}{k}\mu_\phi(v) - \abs{w}^2 \;,
\end{equation*}
satisfies $i_{X_{\hat\tau}}\Omega = dH_{\hat{\tau}}$.  It follows that
if we restrict to a regular level set of $H_{\hat{\tau}}$, say
$H_{\hat{\tau}}^{-1}(\epsilon/k)$, we have $i_{X_{\hat\tau}}\Omega=0$.
In other words, on such a level set the generator of the
$\hat\tau$--action is contained in the kernel of the $2$--form.  In
fact, this kernel contains the subspace spanned by the infinitesimal
generators of the $G$--, $\phi$--, and $\hat{\tau}$--actions.  For an
arbitrary skew--symmetric two--form $\omega$ on a vector space $V$ and
a linear subspace $W\leq V$, one has
\begin{equation*}
  \dim W+\dim W^{\omega}=\dim V+\dim(W\cap \ker\omega) \;.
\end{equation*}
With $V=T(\nu_k\times\C)$ and $W=TH_{\hat{\tau}}^{-1}(\epsilon/k)$,
using that $\ker\Omega$ is spanned by the generators of the $G$-- and
$\phi$--actions, this formula gives $\dim W^\Omega =\dim G+2$, which
implies that on $H_{\hat{\tau}}^{-1}(\epsilon/k)$ the kernel of
$\left.\Omega\right|_{W}$ coincides with the subspace spanned by the
infinitesimal generators of the $G$--, $\phi$--, and
$\hat{\tau}$--actions.  Hence the quotient $P_{\epsilon} :=
H_{\phi}^{-1}(\epsilon/k)/\hat{\tau}$, with the structure induced by
$\Omega$, $\phi$, and $G$ is a smooth presymplectic $\S^1\times
G$--manifold.

Notice that $H_{\hat{\tau}}^{-1}(\epsilon/k)$ can be written as the
disjoint union of two $\hat{\tau}$--invariant manifolds
\begin{equation*}
  H_{\hat{\tau}}^{-1}(\epsilon/k)= \left\{(v,w)\,\left|\,
      \mu_\phi(v) > \epsilon,\ \abs{w}^2 =
      \frac{\mu_\phi(v)-\epsilon}{k}\right. \right\}\sqcup
  \bigl\{(v,0)\,\bigm|\,\mu_\phi(v)=\epsilon \bigr\} \;.
\end{equation*}
Choose $\delta >0$ such that $\mu_\phi^{-1}(\delta)$ is contained in
$U$, the neighbourhood of $P_k$ constructed in Proposition~\ref{propo:
  Morse Bott}.  Notice that $\mu_\phi^{-1}(\delta)$ has the structure
of a sphere bundle over~$P_k$.  For $0<\epsilon<\delta$, denote by
$\nu_k(\epsilon)$ the subset of $\nu_k$ given by $\{\mu_\phi(v) <
\epsilon\}$, by $\nu_k(\epsilon, \delta)$ the ``annulus'' $\bigl\{
v\in\nu_k \bigm|\, \epsilon < \mu_\phi(v) < \delta \bigr\}$, and
consider the map
\begin{equation*}
  \Phi:\, \nu_k(\epsilon, \delta) \to P_{\epsilon}\ ,
  \qquad v\mapsto 
  \left[v, \sqrt{\frac{\mu_\phi(v)-\epsilon}{k}}\right] \;.
\end{equation*}
This is an $\S^1\times G$--equivariant diffeomorphism (onto its
image), where the $\S^1$--action is the $\beta$--action on $\nu_k$ and
the $\phi$--action on $P_{\epsilon}$.  Its inverse can be constructed
as follows: given $[v,w]$ with $w\neq 0$, we first represent the same
class by an element $(v', w')$ such that $w'$ is a real positive
number, and then define
\begin{equation*}
  \Phi^{-1}([v,w]):= v' \;.
\end{equation*}
Moreover, since $\Phi$ factors through the map $\nu_k(\epsilon,
\delta) \hookrightarrow H_{\hat{\tau}}^{-1}(\epsilon/k)$ which is the
identity in the first component and a real function in the second one,
we have
\begin{equation*}
  \Phi^*\Omega = \Phi^*(\omega, -i\,dw\wedge d\bar{w})=\omega \;,
\end{equation*}
hence $\Phi$ gives in fact an equivariant presymplectic
identification of $\nu_k(\epsilon, \delta)$ with its image under
$\Phi$.  More precisely we have
\begin{equation*}
  \Phi(\nu_k(\epsilon, \delta))  =\bigl\{[v,w]\in P_{\epsilon}\,\bigm|\:
      \epsilon<\mu_\phi(v) < \delta \bigr\}  \;.
\end{equation*}
We can now remove a tubular $\epsilon$--neighbourhood of $P_k$ in
$\nu_k$ and glue in the smooth manifold
\begin{equation*}
  V(\delta):=
  \left.\left\{(v,w)\,\left|\:\epsilon\leq\mu_\phi(v) < \delta,\,
        \abs{w}^2=\frac{\mu_\phi(v)-\epsilon}{k}
      \right.\right\} \right/\hat{\tau}
\end{equation*}
along the open ``collar'' $\nu_k(\epsilon, \delta)$, using the map
$\Phi$.  In this way we define the new manifold
\begin{equation*}
  \widetilde{P} =\bigl(P-\overline{\nu_k(\epsilon)}\bigr)\bigcup_{\Phi} 
  V(\delta) \;.
\end{equation*}
Since $\Phi$ is equivariant, the $\beta$--action on
$P-\overline{\nu_k(\epsilon)}$ and the $\phi$--action on $V(\delta)$
fit together to give a circle action $\tilde{\beta}$ on
$\widetilde{P}$, which by construction coincides with $\beta$ outside
a $\delta$--neighbourhood of $P_k$.  Similarly, the $G$--actions can
be combined to define a $G$--action on $\widetilde{P}$.  Moreover,
$\Phi$ identifies the given closed $2$--forms on the two sides of the
gluing, so $\widetilde{P}$ also admits a closed $2$--form
$\widetilde{\omega}$ with the property that $\widetilde{\omega}
=\omega_P$ on $P-\nu_k(\delta)$.  With the action of $\S^1\times G$
and the $2$--form $\widetilde{\omega}$ just defined, $\widetilde{P}$
is a presymplectic manifold.

Moreover, there exists a map $f:\,\widetilde{P}/\tilde{\beta}\to
P/\beta$, which is a $G$--equivariant presymplectic orbifold
isomorphism outside an arbitrarily small neighbourhood of $P_k$ (and
in fact coincides with the identity map outside a slightly larger
neighbourhood).  We shall describe how to define $f$.  On
$\bigl(P-\overline{\nu_k(\delta)}\bigr)/\beta$ it is simply the
identity.  In order to define it on $V(\delta)/\phi$ a little more
work is required.  First of all, denote by $S_{\epsilon}$ the quotient
$\bigl\{(v,0)\in H_{\hat{\tau}}^{-1} (\epsilon)\,\bigm|\,
\mu_\phi(v)=\epsilon \bigr\}/\hat\tau$.  Then the inverse of the
gluing map $\Phi$ gives us a diffeomorphism $\Phi^{-1}:\,
V(\delta)-S_{\epsilon}\to \nu_k(\epsilon, \delta)$.  Since $\Phi$ is
equivariant with respect to the $\phi\times G$-- and $\beta\times
G$--actions, this descends to a $G$--equivariant presymplectic
orbifold isomorphism $\bigl(V(\delta) - S_\epsilon\bigr)/\phi\to
\nu(\epsilon, \delta)/\beta$.  Let $h:[0, \infty)\to [0, \infty)$ be a
smooth monotone real function, which satisfies the following three
conditions:
\begin{itemize}
\item[(i)] $h(t)=0$ for all $t\leq \epsilon$;
\item[(ii)] $h$ is strictly increasing on $(\epsilon,\delta)$;
\item[(iii)] $h(t)=1$ for all $t\geq \delta$.
\end{itemize} 
For $v\in \nu_k(\epsilon, \delta)$, define
\begin{equation*}
  \widetilde{S}:\, \nu_k(\epsilon, \delta)\to \nu_k(\delta)-P_k\ ,
\qquad v\mapsto h\bigl(\mu_\phi(v)\bigr)\cdot v \;.
\end{equation*}
The map $\widetilde{S}$ is $\beta\times G$--equivariant so it descends
to a well-defined $G$--equivariant map
\begin{equation*}
S: \nu_k(\epsilon, \delta)/\beta\to
(\nu_k(\delta)-P_k)/\beta \;.
\end{equation*}

The composition of $\Phi^{-1}$ with the ``stretching'' map $S$ yields
a map
\begin{equation*}
  f:\;(V(\delta)-S_{\epsilon})/\phi\to (\nu_k(\delta)-P_k)/\beta\;.
\end{equation*}
Because of the boundary conditions on $h$, the map $f$ glues on the
outer side with the identity map on $(P-\nu_k(\delta))/\beta$. On the
inner side we extend it as the map $S_{\epsilon}/\phi\to P_k/\beta$,
which sends the $\phi$--orbit of $[v,0]$ to the $\beta$--orbit of
$\pi(v)$ ($\pi$ being the projection of $\nu_k$ to its zero section).
To see that $f$ is continuous in a neighbourhood of $S_\epsilon/\phi$,
one has to show that for any sequence $[v_k,w_k]/\phi \subset
V(\delta)/\phi$ that converges to some element $[v,0]/\phi$, it
follows that $f\bigl([v_k,w_k]/\phi\bigr)$ converges to $f
\bigl([v,0]/\phi\bigr)$.  We can find representatives
$(v_k^\prime,w_k^\prime) \in H_{\hat{\tau}}^{-1}(\epsilon/k)$ for the
sequence that converge to $(v^\prime,0)$, and hence
$f\bigl([v_k^\prime,w_k^\prime]/\phi\bigr) = \bigl(h
\bigl(\mu_\phi(v_k^\prime) \bigr)\cdot v_k^\prime\bigr)/\beta$
converges to $\pi(v^\prime) /\beta = \pi(v)/\beta$.

\section{Resolution of symplectic cuts}

The methods described above can also be applied in some situations to
find resolutions of symplectic cuts, thus giving a (positive) answer to a
question posed by River Chiang and Eugene Lerman.

First note that the construction in the proof of
Lemma~\ref{equivariant_circle_resolution} can also be carried out so
as to preserve an additional $K$--action.  If $(P,\omega)$ is a
presymplectic $\S^1\times G$--manifold as in
Lemma~\ref{equivariant_circle_resolution} with an additional
$K$--action that preserves $\omega$ and commutes with $\S^1\times G$,
then for an arbitrarily small neighborhood $U$ of the singular set of
the circle action we find a presymplectic $\S^1\times G$--manifold
$(\widetilde P,\widetilde\omega)$ with a $K$--action such that the
$\S^1$--action is free, and such that there is a map $p:\,\widetilde
P/\S^1 \to P/\S^1$ that is a $G$--equivariant presymplectic resolution
of $P/\S^1$ supported on $U$.  The projection $p$ is also
$K$--equivariant, and if the initial $K$--action was ``Hamiltonian''
with a moment map $\mu_K$ such that $\iota_{Y_M}\omega = d
\pairing{\mu_K}{Y}$ for every $Y \in\mathfrak{k}$, then we will also
recover a moment map $\widetilde\mu_K$ on the resolution that
coincides with $\mu_K$ outside the neighborhood $U$.

To simplify the presentation in the proof, we chose not to include
such a $K$--action: With this additional ingredient, most steps of the
proof just go through in a straightforward way, but let us remark that
in Proposition~\ref{propo: Morse Bott}, the function $\mu_\phi$ is
$K$--invariant, because
\begin{equation*}
  d\bigl(\lie{Y_P}\mu_\phi\bigr) = \lie{Y_P}
  \bigl(\iota_{X_\phi}\omega\bigr) = 0\;,
\end{equation*}
for every infinitesimal generator $Y_P$ associated to $Y
\in\mathfrak{k}$.  It follows that $\lie{Y_P}\mu_\phi =
\mathrm{const}$, and since $Y_P$ is tangent to the stratum $P_k$, and
$\mu_\phi$ vanishes on $P_k$, we have $\lie{Y_P}\mu_\phi = 0$.  For
the proof of Proposition~\ref{propo: presymplectic neighbourhood
  theorem} in the appendix, average $\alpha$ also over the group $K$,
and show that $X_s$ commutes with the generators of $K$ in the same
way as it was done for the $G$--action.  The moment map $\widetilde
\mu_K$ on the patch $V(\delta)$ is given by projection onto the
first factor, followed by the original moment map $\mu_K$.

\begin{theorem}
  Let $(M,\omega)$ be a Hamiltonian $\S^1$--manifold with Hamiltonian
  function $h:\, M\to \R$.  Assume that $a\in h(M)$ is a regular value
  of $h$.  Then we can perform the symplectic cut at $a$ to obtain the
  symplectic orbifolds
  \begin{equation*}
    M_{(-\infty,a]} \quad\text{ and }\quad M_{[a,\infty)}\;.
  \end{equation*}
  For any $\epsilon>0$, there are symplectic resolutions of both
  orbifolds by symplectic $\S^1$--manifolds $(\widetilde
  M_-,\omega_-)$ and $(\widetilde M_+,\omega_+)$ with Hamiltonians
  $\widetilde H_-$ and $\widetilde H_+$, respectively, such that
  $h^{-1}(-\infty,a-\epsilon)$ is isomorphic to $\widetilde
  H_-^{-1}(-\infty,a-\epsilon)$ and $h^{-1}(a+\epsilon,\infty)$ to
  $\widetilde H_+^{-1}(a+\epsilon,\infty)$.
\end{theorem}
\begin{proof}
  The symplectic cut is performed by taking the product manifold
  $M\times \C$ with symplectic form $\omega \oplus -i dz\wedge d\bar
  z$ and circle action $e^{i\phi} *(p,z) := (e^{i\phi}p,e^{i\phi}z)$
  or $e^{i\phi} *(p,z) := (e^{i\phi}p,e^{-i\phi}z)$, depending on
  whether one wants to produce $M_{(-\infty,a]}$ or $M_{[a,\infty)}$.
  In the first case, the Hamiltonian function will be $H_-(p,z) = H(p)
  + \abs{z}^2$, in the second it will be $H_+(p,z) = H(p) -
  \abs{z}^2$.  Taking the symplectic reduction at the value $a$ will
  yield the symplectic cut spaces.

  Let us now focus on the `positive' part of the cut,
  $M_{[a,\infty)}$, since the argument for the negative part is
  completely analogous.  There is a symplectic embedding of
  $h^{-1}(a+\epsilon,\infty)$ into $M_{[a,\infty)}$, defined by $p
  \mapsto \bigl[p, \sqrt{h(p)-a}\bigr]$: it is equivariant with
  respect to the circle action induced on $M_{[a,\infty)}$ by
  $e^{i\phi} *(p,z) := (e^{i\phi}p, z)$.  Notice that this action on
  the presymplectic manifold $H_+^{-1}(a)$ plays the role of the
  additional $K$--action mentioned above.
 
  The potential orbifold singularities of $M_{[a,\infty)}$ lie in the
  cut hypersurface $\bigl\{[p,0]\bigm|\, p \in h^{-1}(a)\bigr\}$.  Let
  $U_{\epsilon}$ be the neighbourhood of this set defined as the
  complement of the image of $\bigl\{(p,z) \bigm|\, h(p)\geq a +
  \epsilon, \abs{z}=\sqrt{h(p)-a}\bigr\}$.  Denote by $\widetilde M_+$
  the resolution of $M_{[a,\infty)}$ supported on $U_\epsilon$
  obtained with the method of the previous section and preserving the
  additional circle action (the $K$--action).  Then
  $h^{-1}(a+\epsilon,\infty)$ embeds into $\widetilde M_+$ in a
  symplectic and equivariant way.  By construction, the Hamiltonian
  for the $K$--action on $\widetilde M_+$ coincides with $h$ outside
  the neighbourhood $U_{\epsilon}$, so the image of
  this embedding coincides with $\widetilde
  H_+^{-1}(a+\epsilon,\infty)$.
\end{proof}

\appendix 

\section{Proof of Proposition~\ref{propo: presymplectic neighbourhood
    theorem}}

The restrictions of $\omega$ and $\omega_P$ coincide along the zero
section of $\nu_k$.  In particular, if we denote by $i_0:\,
P_k\hookrightarrow \nu_k$ the inclusion of $P_k$ as the zero section
of $\nu_k$, we have
\begin{equation*}
  i_0^* (\omega -\omega_P) = 0 \;.
\end{equation*}
This implies that there exists a $1$--form $\alpha$ on $\nu_k$ such
that $\omega-\omega_P = d\alpha$ and moreover $\alpha$ vanishes on
$T_x \nu_k$ for every $x\in P_k$ (see for example
\cite[Theorem~6.8]{CannasSilva} or
\cite[Lemma~3.14]{McDuffSalamonIntro}).  Furthermore, we may assume
$\alpha$ to be $\beta$-- and $G$--invariant, because if it were not,
we could replace it by its average
\begin{equation*}
  \alpha^\prime = \int_{\S^1\times G} \left(g^*\lambda^*\alpha\right)
  \, d\lambda\,dg\;,
\end{equation*}
which still satisfies
\begin{align*}
  d\alpha^\prime = \int_{\S^1\times G} \bigl(g^*\lambda^*d\alpha\bigr)
  \,d\lambda\, dg = \int_{\S^1\times G}
  g^*\lambda^*(\omega-\omega_P)\,d\lambda\, dg = \omega - \omega_P
\end{align*}
and $\alpha^\prime = 0$ on points of $P_k$.  Let $Y$ be an arbitrary
element in the Lie algebra of $\S^1\times G$, and let $Y_P$ be its
infinitesimal generator.  Notice that from the $\S^1\times
G$--invariance, we also obtain that $i_{Y_P}\alpha = \mathrm{const}$,
but since $i_{Y_P}\alpha = 0$ on the zero section, it follows that
$Y_P$ lies everywhere in the kernel of $\alpha$.

Now define the $1$--parameter family of presymplectic $2$--forms
\begin{equation*}
  \omega_s := s\,\omega + (1-s)\,\omega_P \;.
\end{equation*}
Assume there exists a time-dependent vector field $X_s$ such that
$\omega = \left(\Phi_s^{X_s}\right)^*\omega_s$, where $\Phi_s^{X_s}$
denotes the flow of $X_s$.  Then we have
\begin{equation*}
  0 = \frac{d}{ds} \left(\Phi_s^{X_s}\right)^*\omega_s =
  \left(\Phi_s^{X_s}\right)^* \Bigl(\lie{X_s}\omega_s +
  \frac{d}{ds}\omega_s\Bigr)
\end{equation*}
which is equivalent to $\lie{X_s}\omega_s + d\alpha = 0$, or
\begin{equation*}
  d\bigl(i_{X_s}\omega_s + \alpha\bigr) = 0\;.
\end{equation*}
In order to show that $\omega$ and $\omega_P$ are $\S^1\times
G$--equivariantly isomorphic, we need to find $X_s$ satisfying the
last equation, integrating to a time--$1$ flow in a neighbourhood of
$P_k$ and commuting with~$\S^1\times G$.

Let $X_s$ be the unique vector field determined by
\begin{equation*}
  \begin{cases}
    g(Y_P, X_s) = 0 \Leftrightarrow X_s\in\Omega^P =
    (\ker\omega_P)^\perp\\
    \left.\bigl(i_{X_s}\omega_s + \alpha\bigr)\right|_{\Omega^P} = 0
  \end{cases}
\end{equation*}
for any infinitesimal generator $Y_P$ for the $\S^1\times G$--action.
It is easy to see that $i_{Y_P}\bigl(i_{X_s}\omega_s + \alpha\bigr) =
0$, so it follows that $i_{X_s}\omega_s+\alpha = 0$ on the whole
tangent space and not just on $\Omega^P$.  To see that $X_s$ commutes
with $X_\beta$, compute
\begin{align*}
  0 &= \lie{Y_P} \bigl(i_{X_s}\omega_s + \alpha\bigr) =
  i_{\lie{Y_P}X_s}\omega_s \intertext{and} 0 &=
  \lie{Y_P}\bigl(g(X_s,Y_P)\bigr) = g(Y_P, \lie{Y_P}X_s) \;.
\end{align*}
Combining this with the fact that $\omega_s$ is nondegenerate on
$\Omega^P$ we get that $\lie{Y_P}X_s = 0$.

Finally notice that, since $\alpha$ vanishes at all points of the zero
section of $\nu_k$, $X_s$ also does and hence its flow is defined in a
neighbourhood of this section up to time one.

\bibliographystyle{amsalpha}

\bibliography{main}


\end{document}